\documentclass[a4paper,12pt]{article}
\usepackage{amsmath,amssymb,amsthm,cite}

\makeatletter
\newcommand{\subjclassname@later}{\textup{2010} Mathematics Subject Classification}
\makeatother

\newtheorem{theorem}{Theorem}[section]
\newtheorem{lemma}{Lemma}[section]

\newtheorem{remark}{Remark}[section]

\numberwithin{equation}{section} \numberwithin{theorem}{section}

\DeclareMathOperator{\RE}{Re}

\allowdisplaybreaks

\begin{document}
\title{Applications of  Differential Subordination for Functions with Fixed
Second  Coefficient to Geometric Function Theory
}
 \author{See Keong Lee \footnote{Email: sklee@cs.usm.my } \\ \em
School of Mathematical Sciences, Universiti Sains Malaysia,\\ \em
11800 USM,  Penang,  Malaysia\\  and \\
 V.  Ravichandran \footnote{Email: vravi@maths.du.ac.in }\\ \em
Department of Mathematics, University of Delhi,\\ \em Delhi 110
007, India\\  and \\
Shamani Supramaniam \footnote{Email: sham105@hotmail.com}\\ \em
School of Mathematical Sciences,
Universiti Sains Malaysia,\\ \em 11800 USM,  Penang,  Malaysia}
 \maketitle

\begin{abstract} The theory of  differential subordination developed by S. S. Miller and P. T. Mocanu [Differential Subordinations, Dekker, New York, 2000] was recently extended to functions with fixed initial coefficient by  R. M. Ali, S. Nagpal\ and\ V. Ravichandran [Second-order differential subordination for analytic functions with fixed initial coefficient, \emph{Bull.\ Malays.\ Math.\ Sci.\ Soc.\ (2)} {\bf 34} (2011), 611--629] and applied to obtain several generalization of classical results in geometric function theory.  In this paper, further applications of this subordination theory is given. In particular,  several sufficient conditions related to starlikeness, convexity, close-to-convexity  of normalized analytic functions  are derived. Connections with previously known results are indicated.
\end{abstract}


\section{Introduction}
For univalent functions $f(z)=z+\sum_{n=2}^\infty a_{n}z^{n}$ defined on  $\mathbb{D}:=\{ z \in \mathbb{C} : |z| < 1 \}$, the famous Bieberbach theorem shows that $|a_{2}|\leq2$ and this bound for the second coefficient yields the growth and distortion bounds as well as covering theorem. In view of the influence of the second coefficient in the properties of univalent functions, several authors have investigated functions with fixed second coefficient. For a brief survey of the various developments, mainly  on  radius problems, from 1920 to this date, see the recent work by Ali {\it et al.}\  \cite{alifil}.  The theory of first-order differential subordination was developed by Miller and Mocanu  and a very comprehensive account of the theory and numerous application can be found in their monograph \cite{millmoc}.  Ali {\it et al.}\  \cite{alibul} have extended this well-known theory of differential subordination to the functions with preassigned second coefficients. Nagpal and Ravichandran \cite{nagapm} have applied the results in \cite{alibul} to obtain several extensions of well-known results to the functions with fixed second coefficient. In this paper, we continue their investigation by deriving several sufficient conditions for starlikeness of functions with fixed second coefficient.

For convenience, let $\mathcal{A}_{n,b}$ denote the class of all functions $f(z)=z+bz^{n+1}+a_{n+2}z^{n+2}+\cdots$ where $n\in \mathbb{N}=\{1,2,\dotsc\}$ and $b$ is a fixed non-negative real number.
For a fixed $\mu\geq0$, let $\mathcal{H}_{\mu,n}$ consists of analytic functions $p$  on $\mathbb{D}$ of the form \begin{equation}\label{eqp}p(z)=1+\mu z^n +p_{n+1}  z^{n+1}+\cdots, \quad n\in \mathbb{N}\end{equation}
Let $\Omega$ be a subset of $ \mathbb{C}$ and the class $\Psi_{\mu,n}[\Omega]$ consists of those functions $\psi:\mathbb{C}^2\rightarrow \mathbb{C}$ that are continuous in a domain $D\subset \mathbb{C}^2$ with $(1,0)\in D$, $\psi(1,0)\in \Omega$, and satisfy the admissibility condition: $\psi(i\rho,\sigma)\not\in \Omega $  whenever $ (i\rho,\sigma)\in D$, $\rho\in \mathbb{R}$, and
\begin{equation}\label{sigma2}\sigma \leq-\frac{1}{2}\left(n+\frac{2-\mu}{2+\mu}\right)(1+\rho^2).\end{equation} When $\Omega=\{w:\RE w>0\}$, let $  \Psi_{\mu,n}:=\Psi_{\mu,n}[\Omega]$. The following theorem is needed to prove our main results.

\begin{theorem} {\rm\cite[Theorem 3.4]{alibul}}\label{thmsumit}\
Let $p\in \mathcal{H}_{\mu,n}$ with $0<\mu\leq2$.
Let $\psi\in \Psi_{n,\mu}$ with associated domain $D$. If $(p(z),zp'(z))\in D$ and $\RE \psi(p(z),zp'(z))>0$, then $\RE p(z)>0$ for $z\in\mathbb{D}$.
\end{theorem}

 For $\alpha\neq 1$, let
\[ \mathcal{S}^*(\alpha):=\left\{f\in\mathcal{A}: \frac{zf'(z)}{f(z)} \prec \frac{1+(1-2\alpha)z}{1-z}\right\}.\]
The function $p_\alpha(z):=(1+(1-2\alpha)z)/(1-z)$ maps $\mathbb{D}$ onto $\{w\in\mathbb{C}:\RE w >\alpha\}$ for $\alpha<1$ and
onto  $\{w\in\mathbb{C}:\RE w <\alpha\}$ for $\alpha>1$. Therefore, for $\alpha<1$,  $\mathcal{S}^*(\alpha)$ is the class of starlike functions of order $\alpha$ consisting of functions $f\in\mathcal{A}$ for which $\RE(zf'(z)/f(z))>\alpha$.  For $\alpha>1$, $\mathcal{S}^*(\alpha)$ reduces to the class $\mathcal{M}(\alpha)$ consisting of  $ f\in\mathcal{A}$ satisfying $\RE(zf'(z)/f(z))  < \alpha$.  The latter  class $\mathcal{M}(\alpha)$ and its subclasses were investigated in   \cite{rmavr,owa2,vrkgs,ural,ural2}.   For $0\leq\alpha<1$,  $\mathcal{S}^*(\alpha)$  consists of only univalent functions while for other values of $\alpha$, the class contains non-univalent functions. Other classes can be unified in a similar manner by subordination.

 Motivated by the works of   Lewandowski,  Miller\ and  Z\l otkiewicz \cite{Lewandowski1}, several authors \cite{liutam,liumcm,nishi,owa2,Nunokawa1,Obradovich,Padmanabhan1,KSP,Ravi2,singhaml,yangcma} have investigated the functions $f$ for which $zf'(z)/f(z)\cdot$ $(\alpha z f''(z)/f'(z)+1)$ lies in certain region in the right half-plane.
For $\alpha\geq0$  and $\beta<1$,  Ravichandran {\it et al.}\  \cite{ravijipam} have shown that a function $f$ of the form $f(z)=z+a_{n+1}z^{n+1}+\cdots$ satisfying \begin{equation}\label{eqzf}\RE \left(\frac{zf'(z)}{f(z)}\left(\alpha\frac{zf''(z)}{f'(z)}+1\right)\right)>\alpha\beta\left(\beta+\frac{n}{2}-1\right)+\beta-\frac{\alpha n}{2} \end{equation}
 is starlike of order $\beta$.  In the first result of Theorem \ref{th2.7}, we obtain the corresponding result for $f\in \mathcal{A}_{n,b}$.

For function $p$ of the form $p(z)=1+p_1 z+p_2 z^2+\cdots$, Nunokawa {\it et al.}\  \cite{nunoind} showed that  for analytic function $w$, $\alpha p^2(z)+\beta zp'(z)\prec w(z)$ implies $\RE p(z)>0$, where $\beta>0$, $\alpha\geq -\beta/2$. See also \cite{jayafjms}.  Lemma \ref{th2.6} investigate the conditions for similar class of functions.

For complex numbers $\beta$ and $\gamma$ with $\beta\neq 0$, the differential subordination  \[q(z)+\frac{zq'(z)}{\beta q(z)+\gamma}\prec h(z),\] where $q$ is analytic and $h$ is univalent with $q(0) = h(0)$, is popularly known
as Briot-Bouquet differential subordination.
 This particular differential subordination has a significant number of important applications in the theory of analytic functions (for details
see \cite{millmoc}).  The importance of Briot-Bouquet differential subordination inspired many researchers
to work in this area and many generalizations and extensions of the Briot-Bouquet differential subordination have recently been obtained. Ali {\it et al.}\  \cite{alijmaa} obtained several results related to the Briot-Bouquet differential subordination. In Lemmas \ref{th2.2} and \ref{th2.5}, the  Briot-Bouquet differential subordination is investigated for functions with fixed second coefficient.

\section{Subordinations for starlikeness and univalence}
 For $\beta\neq 1$, Theorem~\ref{th2.7} provides several sufficient conditions for $f\in \mathcal{S}^*(\beta)$; in particular, for $0\leq \beta<1$, these are sufficient conditions for  starlikeness of order $\beta$.  Theorem~\ref{th2.7new} is  the meromorphic analogue of Theorem~\ref{th2.7}.
 Theorem~\ref{th2.8} gives sufficient conditions for the subordination $f'(z)\prec (1+(1-2\beta)z)/(1-z)$ to hold. For $\beta=0$, this latter condition is sufficient for the close-to-convexity and hence univalence of the function $f$.

\begin{theorem}\label{th2.7}
Let $\alpha\geq0$, $\beta\neq 1$,   and $0\leq\mu=nb\leq2$. Let $\delta_1$, $\delta_2$, $\delta_3$ and $\delta_4$ be given by
 \begin{align*}
 \delta_1&=-\frac{\alpha}{2}(1-\beta)\left(n+\frac{2-\mu}{2+\mu}\right)
+(1-\alpha)\beta+\alpha\beta^2,\\
\delta_2&=-\frac{1}{2}(1-\beta)\left(n+\frac{2-\mu}{2+\mu}\right)+\beta,\\
\delta_3&=\left\{
                                                           \begin{array}{ll}
                                                             \frac{-\alpha\beta}{2(1-\beta)}\left(n+\frac{2-\mu}{2+\mu}\right)+\beta, & if \quad \hbox{$\beta\leq \frac{1}{2}$,} \\[5pt]
                                                             \frac{-\alpha}{2\beta}(1-\beta)\left(n+\frac{2-\mu}{2+\mu}\right)+\beta, & if \quad \hbox{$\frac{1}{2}\leq\beta$,}
                                                           \end{array}
                                                         \right. \\
\delta_4&=\left\{
\begin{array}{ll}
    \frac{-\beta}{2(1-\beta)}\left(n+\frac{2-\mu}{2+\mu}\right), & if \quad \hbox{$\beta< \frac{1}{2}$,} \\[5pt]
         \frac{-1}{2\beta}(1-\beta)\left(n+\frac{2-\mu}{2+\mu}\right), & if \quad \hbox{$\frac{1}{2}\leq\beta$}.
     \end{array}
     \right.
 \end{align*}
 If $f\in {\mathcal A}_{n,b}$ satisfies one of the following subordinations
\begin{align}\label{eq2.7}&\frac{zf'(z)}{f(z)}\left(\alpha\frac{zf''(z)}{f'(z)}+1\right)
\prec \frac{1+(1-2\delta_1)z}{ 1-z},\\
&\frac{zf'(z)}{f(z)}\left(2+\frac{zf''(z)}{f'(z)}-\frac{zf'(z)}{f(z)}\right)
\prec \frac{1+(1-2\delta_2)z}{ 1-z},\\
&(1-\alpha)\frac{zf'(z)}{f(z)}+\alpha\left(1+\frac{zf''(z)}{f'(z)}\right)\prec \frac{1+(1-2\delta_3)z}{ 1-z},\\
  &1+\frac{zf''(z)}{f'(z)}-\frac{zf'(z)}{f(z)}
  \prec -\frac{ 2\delta_4 z}{ 1-z}\end{align}
then \[\frac{zf'(z)}{f(z)}\prec \frac{1+(1-2\beta)z}{ 1-z}.\]
\end{theorem}

Our next theorem gives sufficient conditions for meromorphic functions to be starlike in the punctured unit disk $\mathbb{D}^*:=\{z\in\mathbb{C}: 0<|z|<1\}$. Precisely, we consider the class
 $ {\Sigma}_{n,b}$ of all analytic functions defined on $\mathbb{D}^*$ of the form
\[ f(z)=\frac{1}{z}+bz^n +a_{n+1}z^{n+1}+\cdots \quad (b\leq0). \]

\begin{theorem}\label{th2.7new}
Let $\alpha\geq0$, $\beta\neq 1$,   and $0\leq \mu=-(n+1)b\leq2$. Let $\delta_1$, $\delta_2$, $\delta_3$ and $\delta_4$ be given by
 \begin{align*}
 \delta_1&=\frac{\alpha}{2}(1-\beta)\left(n+\frac{2-\mu}{2+\mu}\right)
+(1-\alpha)\beta+\alpha\beta^2,\\
\delta_2&=\frac{1}{2}(1-\beta)\left(n+\frac{2-\mu}{2+\mu}\right)+\beta,\\
\delta_3&=\left\{
                                                           \begin{array}{ll}
                                                             \frac{\alpha\beta}{2(1-\beta)}\left(n+\frac{2-\mu}{2+\mu}\right)+\beta, & if \quad \hbox{$\beta\leq \frac{1}{2}$,} \\[5pt]
                                                             \frac{\alpha}{2\beta}(1-\beta)\left(n+\frac{2-\mu}{2+\mu}\right)+\beta, & if \quad \hbox{$\frac{1}{2}\leq\beta$,}
                                                           \end{array}
                                                         \right. \\
\delta_4&=\left\{
\begin{array}{ll}
    \frac{-\beta}{2(1-\beta)}\left(n+\frac{2-\mu}{2+\mu}\right), & if \quad \hbox{$\beta< \frac{1}{2}$,} \\[5pt]
         \frac{-1}{2\beta}(1-\beta)\left(n+\frac{2-\mu}{2+\mu}\right), & if \quad \hbox{$\frac{1}{2}\leq\beta$}.
     \end{array}
     \right.
 \end{align*}
 If $f\in \Sigma_{n,b}$ satisfies one of the following subordinations
\begin{align}\label{eq2.7new}&\frac{zf'(z)}{f(z)}\left(2\alpha-1+\alpha\frac{zf''(z)}{f'(z)}\right)
\prec \frac{1+(1-2\delta_1)z}{ 1-z},\\
&\frac{zf'(z)}{f(z)}\left(\frac{zf''(z)}{f'(z)}-\frac{zf'(z)}{f(z)}\right)
\prec \frac{1+(1-2\delta_2)z}{ 1-z},\\
&-\left((1-\alpha)\frac{zf'(z)}{f(z)}+\alpha\left(1+\frac{zf''(z)}{f'(z)}\right)\right)\prec \frac{1+(1-2\delta_3)z}{ 1-z},\\
  & 1+\frac{zf''(z)}{f'(z)}-\frac{zf'(z)}{f(z)}
  \prec -\frac{  2\delta_4 z}{ 1-z}\end{align}
then \[-\frac{zf'(z)}{f(z)}\prec \frac{1+(1-2\beta)z}{ 1-z}.\]
\end{theorem}

\begin{theorem}\label{th2.8}
Let $\alpha\geq0$, $\beta\neq 1$,  and $0\leq\mu=(n+1)b\leq2$. Let $\delta_1$, $\delta_2$, $\delta_3$ and $\delta_4$ be given as in Theorem 2.1. If $f\in {\mathcal A}_{n,b}$ satisfies one of the following subordinations
\begin{align}\label{eq2.8}
&f'(z)\left[\alpha\left(\frac{zf''(z)}{f'(z)}+f'(z)-1\right)+1\right]\prec \frac{1+(1-2\delta_1)z}{ 1-z}, \\
\label{eq2.12}& f'(z)+zf''(z)
\prec \frac{1+(1-2\delta_2)z}{ 1-z},\\
\label{eq2.13}& \alpha\frac{zf''(z)}{f'(z)}+f'(z)\prec \frac{1+(1-2\delta_3)z}{ 1-z}, \\
 \label{eq2.14} & \frac{zf''(z)}{f'(z)}\prec -\frac{  2\delta_4 z}{ 1-z} \end{align}
then \[f'(z)\prec \frac{1+(1-2\beta)z}{ 1-z}.\]
\end{theorem}

The proof of these  theorems follows from the following series of lemmas.

\begin{lemma}\label{th2.1} Let $\alpha\geq0$, $\beta\neq 1$, $\gamma>0$,  and $0\leq\mu\leq2$. For function $p\in \mathcal{H}_{\mu,n}$ and \[\delta:=-\frac{\gamma}{2}(1-\beta)\left(n+\frac{2-\mu}{2+\mu}\right)+(1-\alpha)\beta+\alpha\beta^2,\] if $p$ satisfies
\begin{equation}\label{eq1}
 (1-\alpha)p(z)+\alpha p^2(z) +\gamma zp'(z)\prec\frac{1+(1-2\delta)z}{1-z}  \end{equation} then \[ p(z)\prec \frac{1+(1-2\beta)z}{1-z} .\]
\end{lemma}

\begin{proof}
 Define the function $q:\mathbb{D}\rightarrow\mathbb{C}$ by $q(z)=(p(z)-\beta)/(1-\beta)$. Then $q$ is analytic and $(1-\beta)q(z)+\beta=p(z)$. By using this, the  inequality \eqref{eq1} can then be written as
 \[\begin{split}\RE\Big[ (1-\beta)(1-\alpha+2\alpha\beta)q(z)+\alpha(1-\beta)^2 q^2(z)\\
 +\gamma(1-\beta)zq'(z)+(1-\alpha)\beta+\alpha\beta^2-\delta\Big]>0.\end{split}\] Define the function $\psi:\mathbb{C}\rightarrow\mathbb{C}$ by \[\psi(r,s)=(1-\beta)(1-\alpha+2\alpha\beta)r+\alpha(1-\beta)^2 r^2+\gamma(1-\beta)s+(1-\alpha)\beta+\alpha\beta^2-\delta.\] For $\rho\in \mathbb{R}$, $n\geq1$ and $\sigma$ satisfying \eqref{sigma2}, it follows that
\begin{align*}
&\RE \psi(i\rho, \sigma)
\\&=\RE\left[(1-\beta)(1-\alpha+2\alpha\beta)i\rho-\alpha(1-\beta)^2 \rho^2+\gamma(1-\beta)\sigma+(1-\alpha)\beta+\alpha\beta^2-\delta \right]
\\&= \gamma(1-\beta)\sigma-\alpha(1-\beta)^2 \rho^2+(1-\alpha)\beta+\alpha\beta^2-\delta
\\&\leq \gamma (1-\beta)\left[-\frac{1}{2}\left(n+\frac{2-\mu}{2+\mu}\right)(1+\rho^2)\right]-\alpha(1-\beta)^2 \rho^2+(1-\alpha)\beta+\alpha\beta^2-\delta
\\&=-\frac{\gamma}{2}(1-\beta)\left(n+\frac{2-\mu}{2+\mu}\right)(1+\rho^2)-\alpha(1-\beta)^2 (\rho^2+1)+\alpha(1-\beta)^2\\&\quad{}+(1-\alpha)\beta+\alpha\beta^2-\delta
\\&=-(1+\rho^2)\left[\frac{\gamma}{2}(1-\beta)\left(n+\frac{2-\mu}{2+\mu}\right)+\alpha(1-\beta)^2\right]
\\&\quad{} +\alpha(1-\beta)^2+(1-\alpha)\beta+\alpha\beta^2-\delta
\\&\leq -\frac{\gamma}{2}(1-\beta)\left(n+\frac{2-\mu}{2+\mu}\right)+(1-\alpha)\beta+\alpha\beta^2-\delta.
\end{align*} Hence $\RE \psi(i\rho, \sigma)\leq 0$. By Theorem \ref{thmsumit}, $\RE q(z)>0$ or equivalently $\RE p(z)>\beta$. \end{proof}


\begin{lemma} \label{th2.2} Let  $\beta\neq 1$, $\gamma>0$,  and $0\leq\mu\leq2$. For function $p\in \mathcal{H}_{\mu,n}$ and \[\delta:=-\frac{\gamma}{2}(1-\beta)\left(n+\frac{2-\mu}{2+\mu}\right)+\beta\] if  $p$  satisfies
\begin{equation*}
 p(z) +\gamma zp'(z)\prec \frac{1+(1-2\delta)z}{1-z}\end{equation*} then \[ p(z)\prec \frac{1+(1-2\beta)z}{1-z}.\]
\end{lemma}

\begin{proof}
Replace $\alpha=0$ in Lemma \ref{th2.1}  to yield the result.
\end{proof}

\begin{lemma} \label{th2.3}Let $\alpha>0$, $\beta\neq1$ and $0\leq\mu\leq2$. Let \[\delta=\left\{
                                                           \begin{array}{ll}
                                                             \frac{-\alpha\beta}{2(1-\beta)}\left(n+\frac{2-\mu}{2+\mu}\right)+\beta:=\delta_2 , & if \quad \hbox{$\beta\leq \frac{1}{2}$,} \\[5pt]
                                                             \frac{-\alpha}{2\beta}(1-\beta)\left(n+\frac{2-\mu}{2+\mu}\right)+\beta:=\delta_1, & if \quad \hbox{$\frac{1}{2}\leq\beta$,}
                                                           \end{array}
                                                         \right. .\] If the function $p\in \mathcal{H}_{\mu,n}$ satisfies
\begin{equation}\label{eq2}
 p(z) +\alpha \frac{zp'(z)}{p(z)}\prec \frac{1+(1-2\delta)z}{1-z}
 \end{equation} then \[ p(z)\prec \frac{1+(1-2\beta)z}{1-z} .\]
\end{lemma}

\begin{proof}
 Similar to the proof of Lemma \ref{th2.1}, let $q:\mathbb{D}\rightarrow\mathbb{C}$ be given by $q(z)=(p(z)-\beta)/(1-\beta)$. Then inequality \eqref{eq2} can be written as
\begin{equation}\label{eq3} \RE \left[(1-\beta)q(z)+\beta+\frac{\alpha(1-\beta)}{(1-\beta)q(z)+\beta}zq'(z)-\delta\right]>0.\end{equation}  Define the function $\psi:\mathbb{C}\rightarrow\mathbb{C}$ by \[\psi(r,s)=(1-\beta)r+\frac{\alpha(1-\beta)}{(1-\beta)r+\beta}s+\beta-\delta.\] Then $\RE\psi(q(z),zq'(z))>0$ and $\RE \psi(1,0)>0$. To show that $\psi\in \Psi_{\mu, n}$, by using \eqref{sigma2}, it follows that
\begin{align*}
\RE \psi(i\rho, \sigma)&=\RE\left[(1-\beta)i\rho+\frac{\alpha(1-\beta)}{(1-\beta)i\rho+\beta}\sigma+\beta-\delta\right]
\\&=\RE\left[(1-\beta)i\rho+\frac{\alpha\beta(1-\beta)}{\beta^2+(1-\beta)^2\rho^2}\sigma-\frac{\alpha(1-\beta)^2 i\rho}{\beta^2+(1-\beta)^2\rho^2}\sigma +\beta-\delta\right]
\\&=\frac{\alpha\beta(1-\beta)}{\beta^2+(1-\beta)^2\rho^2}\sigma+\beta-\delta
\\&\leq \frac{\alpha\beta(1-\beta)}{\beta^2+(1-\beta)^2\rho^2}\left[-\frac{1}{2}\left(n+\frac{2-\mu}{2+\mu}\right)(1+\rho^2)\right]+\beta-\delta
\\&=-\frac{\alpha\beta}{2}(1-\beta)\left(n+\frac{2-\mu}{2+\mu}\right)\left(\frac{1+\rho^2}{\beta^2+(1-\beta)^2\rho^2}\right)+\beta-\delta.
\end{align*}
For ${1}/{2}\leq\beta$, the expression \[\frac{1+\rho^2}{\beta^2+(1-\beta)^2\rho^2}\] attains minimum at $\rho=0$ and therefore \begin{align*}\RE \psi(i\rho, \sigma)&\leq-\frac{\alpha\beta}{2}(1-\beta)\left(n+\frac{2-\mu}{2+\mu}\right)\frac{1}{\beta^2}+\beta-\delta_1
\\&=\frac{-\alpha}{2\beta}(1-\beta)\left(n+\frac{2-\mu}{2+\mu}\right)+\beta-\delta_1.
 \end{align*}
Hence $\RE \psi(i\rho, \sigma)\leq 0$.

For $\beta\leq {1}/{2}$,  \begin{align*}\RE \psi(i\rho, \sigma)&\leq-\frac{\alpha\beta}{2}(1-\beta)\left(n+\frac{2-\mu}{2+\mu}\right)\frac{1}{(1-\beta)^2}+\beta-\delta_2
\\&=\frac{-\alpha\beta}{2(1-\beta)}\left(n+\frac{2-\mu}{2+\mu}\right)+\beta-\delta_2.
 \end{align*} Hence $\RE \psi(i\rho, \sigma)\leq 0$. Thus Theorem \ref{thmsumit} implies $\RE q(z)>0$ or equivalently $\RE p(z)>\beta$.
\end{proof}

%

\begin{lemma} \label{th2.4} Let $\beta\neq 1$ and $0\leq\mu\leq2$. Let \[\delta=\left\{
\begin{array}{ll}
    \frac{-\beta}{2(1-\beta)}\left(n+\frac{2-\mu}{2+\mu}\right), & if \quad \hbox{$\beta< \frac{1}{2}$,} \\[5pt]
       \frac{-1}{2\beta}(1-\beta)\left(n+\frac{2-\mu}{2+\mu}\right)  , & if \quad \hbox{$\frac{1}{2}\leq\beta$,}
     \end{array}
     \right. \quad  (z\in\mathbb{D}).\] If the function $p\in \mathcal{H}_{\mu,n}$  satisfies
\begin{equation*}
  \frac{zp'(z)}{p(z)}\prec - \frac{2\delta z}{1-z} \end{equation*} then \[ p(z)\prec \frac{1+(1-2\beta)z}{1-z}.\]
\end{lemma}

\begin{proof}
Let $q(z)=(p(z)-\beta)/(1-\beta)$ or $(1-\beta)q(z)+\beta=p(z)$. Then
\begin{equation}\label{eq4} \frac{zp'(z)}{p(z)}= \frac{(1-\beta)}{(1-\beta)q(z)+\beta}zq'(z).\end{equation} Define $\psi:\mathbb{C}^2 \rightarrow\mathbb{C}$ by
\[\psi(r,s)=\frac{(1-\beta)}{(1-\beta)r+\beta}s-\delta.\] Then $\psi(r,s)$ is continuous on $\mathbb{C}-\{-\beta/(1-\beta)\}$ and by using \eqref{sigma2}, it follows that
\begin{align*}
\RE \psi(i\rho, \sigma)&=\RE\left(\frac{(1-\beta)}{(1-\beta)i\rho+\beta}\sigma-\delta\right)
\\&=\RE \left(\frac{\beta(1-\beta)}{\beta^2+(1-\beta)^2\rho^2}\sigma-\frac{(1-\beta)^2 i\rho}{\beta^2+(1-\beta)^2\rho^2}\sigma-\delta\right)
\\&=\frac{\beta(1-\beta)}{\beta^2+(1-\beta)^2\rho^2}\sigma-\delta
\\&\leq\frac{\beta(1-\beta)}{\beta^2+(1-\beta)^2\rho^2}\left[-\frac{1}{2}\left(n+\frac{2-\mu}{2+\mu}\right)(1+\rho^2)\right]-\delta
\\&=-\frac{\beta}{2}(1-\beta)\left(n+\frac{2-\mu}{2+\mu}\right)\left(\frac{1+\rho^2}{\beta^2+(1-\beta)^2\rho^2}\right)-\delta.
\end{align*}For ${1}/{2}\leq\beta$,  the expression \[\frac{1+\rho^2}{\beta^2+(1-\beta)^2\rho^2}\] attains its minimum at $\rho=0$ and therefore
\begin{align*}\RE \psi(i\rho, \sigma)&\leq-\frac{\beta}{2}(1-\beta)\left(n+\frac{2-\mu}{2+\mu}\right)\frac{1}{\beta^2}-\delta
\\&=\frac{-1}{2\beta}(1-\beta)\left(n+\frac{2-\mu}{2+\mu}\right)-\delta.
 \end{align*} Hence $\RE \psi(i\rho, \sigma)\leq 0$.

For $\beta\leq {1}/{2}$, \begin{align*}\RE \psi(i\rho, \sigma)&\leq-\frac{\beta}{2}(1-\beta)\left(n+\frac{2-\mu}{2+\mu}\right)\frac{1}{(1-\beta)^2}-\delta
\\&=\frac{-\beta}{2(1-\beta)}\left(n+\frac{2-\mu}{2+\mu}\right)-\delta.
 \end{align*}Hence $\RE \psi(i\rho, \sigma)\leq 0$. Thus Theorem \ref{thmsumit} implies $\RE q(z)>0$ or equivalently $\RE p(z)>\beta$.
\end{proof}

\begin{lemma}\label{th2.5} Let $\alpha>0$, $\beta\neq 1$, and $0\leq\mu\leq2$. Let \[\delta=\left\{
                                                           \begin{array}{ll}
                                                             \frac{-1}{2}\frac{(1-\beta)}{(\alpha\beta+\gamma)}\left(n+\frac{2-\mu}{2+\mu}\right)+\beta, & if \quad \hbox{$\gamma\geq\alpha(1-2\beta)$,} \\[5pt]
                                                            \frac{-1}{2}\frac{(\alpha\beta+\gamma)}{\alpha^2(1-\beta)}\left(n+\frac{2-\mu}{2+\mu}\right)+\beta, & if \quad \hbox{$\gamma\leq\alpha(1-2\beta)$,}
                                                           \end{array}
                                                         \right. \quad  (z\in\mathbb{D}).\]If the function $p\in \mathcal{H}_{\mu,n}$  satisfies
\begin{equation*}
 p(z) +\frac{zp'(z)}{\alpha p(z)+\gamma}\prec \frac{1+(1-2\delta)z}{1-z}
 \end{equation*} then \[ p(z)\prec \frac{1+(1-2\beta)z}{1-z} .\]
\end{lemma}

\begin{proof} Define
 $q(z)=(p-\beta)/(1-\beta)$ or $(1-\beta)q+\beta=p(z)$. Then
\begin{equation}\label{eq5} p(z) +\frac{zp'(z)}{\alpha p(z)+\gamma}=(1-\beta)q(z)+\beta+\frac{(1-\beta)}{\alpha[(1-\beta)q(z)+\beta]+\gamma}zq'(z).\end{equation} Define $\psi:\mathbb{C}^2 \rightarrow\mathbb{C}$ by \[\psi(r,s)=(1-\beta)r+\frac{(1-\beta)}{\alpha(1-\beta)r+\alpha\beta+\gamma}s+\beta-\delta.\] Thus $\psi(r,s)$ is continuous  and using \eqref{sigma2}, it follows that
\begin{align*}
\RE \psi(i\rho, \sigma)&=\RE\left[(1-\beta)i\rho+\frac{(1-\beta)}{\alpha(1-\beta)i\rho+\alpha\beta+\gamma}\sigma+\beta-\delta\right]
\\&=\frac{(1-\beta)(\alpha\beta+\gamma)}{(\alpha\beta+\gamma)^2+\alpha^2(1-\beta)^2\rho^2}\sigma+\beta-\delta
\\&\leq \frac{(1-\beta)(\alpha\beta+\gamma)}{(\alpha\beta+\gamma)^2+\alpha^2(1-\beta)^2\rho^2}\left[-\frac{1}{2}\left(n+\frac{2-\mu}{2+\mu}\right)(1+\rho^2)\right]+\beta-\delta
\\&=\frac{-1}{2}(1-\beta)(\alpha\beta+\gamma)\left(n+\frac{2-\mu}{2+\mu}\right)\left(\frac{1+\rho^2}{(\alpha\beta+\gamma)^2+\alpha^2(1-\beta)^2\rho^2}\right)+\beta-\delta
\end{align*}
For $\gamma\leq\alpha(1-2\beta)$, \begin{align*}
\RE \psi(i\rho, \sigma)&\leq \frac{-1}{2}(1-\beta)(\alpha\beta+\gamma)\left(n+\frac{2-\mu}{2+\mu}\right)\frac{1}{\alpha^2(1-\beta)^2}+\beta-\delta
\\&=\frac{-1}{2}\frac{(\alpha\beta+\gamma)}{\alpha^2(1-\beta)}\left(n+\frac{2-\mu}{2+\mu}\right)+\beta-\delta.
\end{align*} Hence $\RE \psi(i\rho, \sigma)\leq 0$.

For $\gamma\geq\alpha(1-2\beta)$, the expression \[\frac{1+\rho^2}{(\alpha\beta+\gamma)^2+\alpha^2(1-\beta)^2\rho^2}\] attains minimum at $\rho=0$ and therefore \begin{align*}
\RE \psi(i\rho, \sigma)&\leq \frac{-1}{2}(1-\beta)(\alpha\beta+\gamma)\left(n+\frac{2-\mu}{2+\mu}\right)\frac{1}{(\alpha\beta+\gamma)^2}+\beta-\delta
\\&=\frac{-1}{2}\frac{(1-\beta)}{(\alpha\beta+\gamma)}\left(n+\frac{2-\mu}{2+\mu}\right)+\beta-\delta.
\end{align*} Thus $\RE \psi(i\rho, \sigma)\leq 0$ and result follows.
\end{proof}

%

\begin{lemma} \label{th2.6}Let  $\beta\neq 1$, $\gamma>0$,  and $0\leq\mu\leq2$. If the function $p\in \mathcal{H}_{\mu,n}$  satisfies
\begin{equation}\label{eq6}
 p^2(z) +\gamma zp'(z)\prec \frac{1+(1-2\delta)z}{1-z}  \quad (z\in\mathbb{D}) \end{equation} where \[\delta:=-\frac{\gamma}{2}(1-\beta)\left(n+\frac{2-\mu}{2+\mu}\right)+\beta^2\] then \[p(z)\prec \frac{1+(1-2\beta)z}{1-z} .\]
\end{lemma}

\begin{proof}
Define $q(z)=(p(z)-\beta)/(1-\beta)$ or $(1-\beta)q(z)+\beta=p(z)$. Using this it can be shown that inequality \eqref{eq6} can be written as
\[\RE \left[((1-\beta)q(z)+\beta)^2+\gamma(1-\beta)zq'(z)-\delta\right]>0.\]
 Then $\psi(r,s)$ is define by \[\psi(r,s)=[(1-\beta)r+\beta]^2+\gamma(1-\beta)s-\delta.\]  By using \eqref{sigma2},  it follows that
\begin{align*}
&\RE \psi(i\rho, \sigma)
\\&=\RE\left[((1-\beta)i\rho+\beta)^2+\gamma(1-\beta)\sigma-\delta \right]
\\&=-(1-\beta)^2\rho^2+\beta^2+\gamma(1-\beta)\sigma-\delta
\\&\leq\gamma(1-\beta)\left[-\frac{1}{2}\left(n+\frac{2-\mu}{2+\mu}\right)(1+\rho^2)\right]+\beta^2-(1-\beta)^2\rho^2-\delta
\\&=-\frac{\gamma}{2}(1-\beta)\left(n+\frac{2-\mu}{2+\mu}\right)(1+\rho^2)-(1-\beta)^2 (\rho^2+1)+(1-\beta)^2+\beta^2-\delta
\\&=-(1+\rho^2)\left[\frac{\gamma}{2}(1-\beta)\left(n+\frac{2-\mu}{2+\mu}\right)+(1-\beta)^2\right]+(1-\beta)^2+\beta^2-\delta
\\&\leq-\frac{\gamma}{2}(1-\beta)\left(n+\frac{2-\mu}{2+\mu}\right)+\beta^2-\delta.
\end{align*} Hence $\RE \psi(i\rho, \sigma)\leq 0$, and Theorem \ref{thmsumit} implies $\RE q(z)>0$ or equivalently $\RE p(z)>\beta$. \end{proof}

\begin{proof}[Proof of Theorem \ref{th2.7}] For a given function $f\in {\mathcal A}_{n,b}$,
let the function $p:\mathbb{D}\rightarrow \mathbb{C}$ be defined by  $p(z)=zf'(z)/f(z)$. Then computation shows that
$p(z)=1+nbz^n+\cdots \in \mathcal{H}_{\mu,n}$ where $\mu=nb$.
Further calculations  yeild
\begin{align*}&\frac{zf'(z)}{f(z)}\left(\alpha\frac{zf''(z)}{f'(z)}+1\right)=(1-\alpha)p(z)+\alpha p^{2}(z)+\alpha zp'(z),\\
&\frac{zf'(z)}{f(z)}\left(2+\frac{zf''(z)}{f'(z)}-\frac{zf'(z)}{f(z)}\right)=p(z)+zp'(z),\\
&(1-\alpha)\frac{zf'(z)}{f(z)}+\alpha\left(1+\frac{zf''(z)}{f'(z)}\right)=p(z)+\alpha \frac{zp'(z)}{p(z)},\\
& 1+\frac{zf''(z)}{f'(z)}-\frac{zf'(z)}{f(z)}= \frac{zp'(z)}{p(z)}.\end{align*}
Hence the result follows from Lemmas \ref{th2.1}, \ref{th2.2}, \ref{th2.3} and \ref{th2.4}.
\end{proof}

\begin{proof}[Proof of Theorem \ref{th2.7new}]
Let  $f\in \Sigma_{n,b}$, and define the function $p:\mathbb{D}\rightarrow \mathbb{C}$ be defined by $p(0)=1$ and  $p(z)=-zf'(z)/f(z)$ for $z\in \mathbb{D}^*$. Then  $p(z)=1-(n+1)bz^{n+1}+\cdots\in \mathcal{H}_{\mu,n}$ with $\mu=-(n+1)b$.
Simple computations shows that
\begin{align*}&\frac{zf'(z)}{f(z)}\left(2\alpha-1+\alpha\frac{zf''(z)}{f'(z)}\right)=(1-\alpha)p(z)+\alpha p^{2}(z)-\alpha zp'(z),\\
&\frac{zf'(z)}{f(z)}\left(\frac{zf''(z)}{f'(z)}-\frac{zf'(z)}{f(z)}\right)=p(z)-zp'(z),\\
&-\left((1-\alpha)\frac{zf'(z)}{f(z)}+\alpha\left(1+\frac{zf''(z)}{f'(z)}\right)\right)=p(z)-\alpha \frac{zp'(z)}{p(z)},\\
& 1+\frac{zf''(z)}{f'(z)}-\frac{zf'(z)}{f(z)}= \frac{zp'(z)}{p(z)}.\end{align*}
Hence the result follows from Lemmas \ref{th2.1}, \ref{th2.2}, \ref{th2.3} and \ref{th2.4}.
\end{proof}

\begin{proof}[Proof of Theorem \ref{th2.8}]
For $f\in {\mathcal A}_{n,b}$, let the function $p:\mathbb{D}\rightarrow \mathbb{C}$ be defined by
$p(z)=f'(z)$. Then $p(z) =1+(n+1)bz^n +(n+2)a_{n+2} z^{n+1}+\cdots   \in \mathcal{H}_{\mu,n}$ with $\mu=(n+1)b$.
Also, we have the following:
\begin{align*}&f'(z)\left(\alpha\left(\frac{zf''(z)}{f'(z)}+f'(z)-1\right)+1\right)=(1-\alpha)p(z)+\alpha p^{2}(z)+\alpha zp'(z),\\
&f'(z)+\alpha zf''(z)=p(z)+\alpha zp'(z),\\
&\alpha\frac{zf''(z)}{f'(z)}+f'(z)=p(z)+\alpha \frac{zp'(z)}{p(z)},\\
&\frac{zf''(z)}{f'(z)}= \frac{zp'(z)}{p(z)}.\end{align*}
Hence the result follows from Lemmas \ref{th2.1}, \ref{th2.2}, \ref{th2.3} and \ref{th2.4}.
\end{proof}

\begin{remark}
\begin{itemize}\item[]
  \item[(i)] For $\beta=0$, the condition \eqref{eq2.8}--\eqref{eq2.14} gives a sufficient condition for close-to-convexity and hence for univalence.
  \item[(ii)]If $\mu=2$, result \eqref{eq2.7} reduces to \cite[Theorem 2.1]{ravijipam}. If $\mu=2$, and $f'(z)$ is considered as $f(z)/z$, result \eqref{eq2.12} reduces to \cite[Theorem 2.4]{ravijipam}.
Inequality \eqref{eq2.13} reduces to \cite[Theorem 2, p. 182]{singhitsf} in the case when $\mu=2$, $n=1$ and $\beta=1/2$.
Furthermore, if $\mu=2$, $n=1$ and $\beta=(\alpha+1)/2$, result \eqref{eq2.14} reduces to \cite[Theorem 1]{owa}.
\end{itemize}

\end{remark}

\section*{Acknowledgements}  The research of the first and last authors are supported respectively by FRGS grant and   MyBrain MyPhD programme of the Ministry of Higher Education, Malaysia. The authors are thankful to the referee for his useful comments.


\begin{thebibliography}{99}


\bibitem{alijmaa}R. M. Ali, V. Ravichandran\ and\ N. Seenivasagan, On Bernardi's integral operator and the Briot-Bouquet differential subordination, \emph{J. Math.\ Anal.\ Appl.}\ {\bf 324} (2006), no.~1, 663--668.

\bibitem{alifil} R. M. Ali, N. E. Cho, N. Jain and V. Ravichandran, Radii of starlikeness and convexity of functions defined by subordination with fixed second coefficients, {\it Filomat} {\bf26} (2012), no.~3, 553–-561.

\bibitem{rmavr}  R. M. Ali, M. H. Mohd, S. K. Lee and  V. Ravichandran, Radii of starlikeness, parabolic starlikeness and strong starlikeness for Janowski starlike functions with complex parameters,  \emph{Tamsui Oxford J. Math. Sci.} 27 (2011),   no.~3, 253--267.

\bibitem{alibul} R. M. Ali, S. Nagpal\ and\ V. Ravichandran, Second-order differential subordination for analytic functions with fixed initial coefficient, \emph{Bull.\ Malays.\ Math.\ Sci.\ Soc.\ (2)} {\bf 34} (2011), no.~3, 611--629.



\bibitem{Lewandowski1}
Z. Lewandowski, S. Miller\ and\ E. Z\l otkiewicz, Generating functions for some classes of univalent functions, \emph{Proc.\ Amer.\ Math.\ Soc.}\ {\bf 56} (1976), 111--117.

\bibitem{owa2}
J.-L. Li\ and\ S. Owa, Sufficient conditions for starlikeness, \emph{Indian J. Pure Appl.\ Math.}\ {\bf 33} (2002), no.~3, 313--318.


\bibitem{liutam} M.-S. Liu, Y.-Y. Liu and Z.-W. Liu, Properties and characteristics of certain subclass of analytic functions with positive coefficients, \emph{Tamsui Oxford J. Inform.\ Math.\ Sci.}\ {\bf27} (2011), no.~1, 39--60.

\bibitem{liumcm}  M.-S. Liu, Y.-C. Zhu\ and\ H. M. Srivastava, Properties and characteristics of certain subclasses of starlike functions of order $\beta$, \emph{Math.\ Comput.\ Modelling} {\bf 48} (2008), no.~3-4, 402--419.

\bibitem{millmoc} S. S. Miller\ and\ P. T. Mocanu, {\it Differential Subordinations}, Dekker, New York, 2000.

\bibitem{nagapm} S. Nagpal and V. Ravichandran, Applications of theory of differential subordination for functions with fixed initial coefficient to univalent functions, \emph{Ann.\ Polon.\ Math.}\ {\bf105} (2012), 225--238.

\bibitem{nishi} J. Nishiwaki\ and\ S. Owa, Coefficient inequalities for certain analytic functions, \emph{Int.\ J. Math.\ Math.\ Sci.}\ {\bf 29} (2002), no.~5, 285--290.

    \bibitem{nunoind} M. Nunokawa, S. Owa, N. Takahashi and H. Saitoh, Sufficient conditions for Carath\'eodory functions, \emph{Indian J. Pure Appl.\ Math.}\ {\bf 33} (2002), no.~9, 1385--1390.

\bibitem{Nunokawa1}
M. Nunokawa, S. Owa, S. K.  Lee, M. Obradovic, M. K. Aouf, H. Saitoh, A. Ikeda\ and\ N. Koike, Sufficient conditions for starlikeness, \emph{Chinese J. Math.}\ {\bf 24} (1996), no.~3, 265--271.

\bibitem{Obradovich}
M. Obradovi\'c\ and\ S. B. Joshi, On certain classes of strongly starlike functions, \emph{Taiwanese J. Math.} {\bf 2} (1998), no.~3, 297--302.

\bibitem{owa2} S. Owa\ and\ H. M. Srivastava, Some generalized convolution properties associated with certain subclasses of analytic functions,  \emph{ J. Inequal.\ Pure Appl.\ Math.}\ 3 (2002), no.~3, Art. 42, 13~pp.

\bibitem{owa} S. Owa, M. Nunokawa, H. Saitoh and H. M. Srivastava, Close-to-convexity, starlikeness, and convexity of certain analytic functions, \emph{Appl.\ Math.\ Lett.}\ {\bf 15} (2002), no.~1, 63--69.







\bibitem{Padmanabhan1}
K. S. Padmanabhan, On sufficient conditions for starlikeness, \emph{Indian J. Pure Appl.\ Math.}\ {\bf 32} (2001), no.~4, 543--550.


\bibitem{KSP}
C. Ramesha, S. Kumar\ and\ K. S. Padmanabhan, A sufficient condition for starlikeness, \emph{Chinese J. Math.}\ {\bf 23} (1995), no.~2, 167--171.


\bibitem{Ravi2}
V. Ravichandran, Certain applications of first order differential subordination, \emph{Far East J. Math.\ Sci.}\  {\bf 12} (2004), no.~1, 41--51.


\bibitem{jayafjms} V. Ravichandran\ and\ M. Jayamala, On sufficient conditions for Caratheodory functions, \emph{Far East J. Math.\ Sci.}\ {\bf 12} (2004), no.~2, 191--201.

\bibitem{ravijipam} V. Ravichandran, C. Selvaraj\ and\ R. Rajalaksmi, Sufficient conditions for starlike functions of order $\alpha$, \emph{J. Inequal.\ Pure Appl.\ Math.}\ {\bf 3} (2002), no.~5, Art. 81, 6~pp.


\bibitem{vrkgs}  V. Ravichandran, H. Silverman, M. Hussain Khan, K. G. Subramanian, Radius problems for a class of analytic functions,   \emph{Demonstrat.\ Math.}\  39 (2006),   no.~1,   67--74.

\bibitem{singhaml}    S. Singh\ and\ S. Gupta, A differential subordination and starlikeness of analytic functions, \emph{Appl.\ Math.\ Lett.}\ {\bf 19} (2006), no.~7, 618--627.

\bibitem{singhitsf} V. Singh, S. Singh\ and\ S. Gupta, A problem in the theory of univalent functions, \emph{Integral Transforms Spec.\ Funct.}\ {\bf 16} (2005), no.~2, 179--186.

\bibitem{ural} B. A. Uralegaddi, M. D. Ganigi\ and\ S. M. Sarangi, Univalent functions with positive coefficients, \emph{Tamkang J. Math.}\ 25 (1994), no.~3, 225--230.

\bibitem{ural2} B. A. Uralegaddi\ and\ A. R. Desai, Convolutions of univalent functions with positive coefficients, \emph{Tamkang J. Math.}\ {\bf 29} (1998), no.~4, 279--285.

\bibitem{yangcma} D. Yang, S. Owa\ and\ K. Ochiai, Sufficient conditions for Carath\'eodory functions, \emph{Comput.\ Math.\ Appl.}\ {\bf 51} (2006), no.~3-4, 467--474.

\end{thebibliography}
\end{document}